\documentclass[reqno,a4paper,11pt]{amsart}

\usepackage{mystyle}
\usepackage{etoolbox}

\usepackage[headings]{fullpage}
\usepackage{parskip}

\usepackage[utf8]{inputenc}
\usepackage[T1]{fontenc}
\usepackage[british]{babel}
\usepackage[pdfencoding=auto]{hyperref}
\hypersetup{colorlinks=true, linkcolor=blue!30!black, citecolor=green!30!black, urlcolor=red!45!black}

\usepackage{microtype}
\usepackage{booktabs,centernot}

\usepackage{amsfonts,amssymb,bbm,mathrsfs,stmaryrd}
\usepackage[scr=boondoxo,scrscaled=1.05]{mathalfa}
\usepackage{amsmath}

\usepackage{mathtools}
\usepackage[noabbrev,capitalize]{cleveref}

\usepackage[breakable]{tcolorbox}

\usepackage[shortlabels]{enumitem}

\usepackage{tikz}
\usetikzlibrary{matrix}
\usetikzlibrary{positioning}
\usetikzlibrary{calc}
\usetikzlibrary{arrows}
\tikzset{> =stealth}
\usetikzlibrary{arrows.meta}
\tikzset{normalHead/.tip={Triangle[open,angle=60:4pt]},}
\tikzset{normalTail/.tip={Triangle[reversed,open,angle=60:4pt]},}
\usetikzlibrary{shapes.geometric}

\theoremstyle{plain}
\newtheorem{theorem}{Theorem}[section]
\newtheorem{lemma}[theorem]{Lemma}
\newtheorem{proposition}[theorem]{Proposition}

\theoremstyle{definition}
\newtheorem{definition}[theorem]{Definition}

\newtheorem{example}[theorem]{Example}

\theoremstyle{remark}
\newtheorem{remark}[theorem]{Remark}

\renewcommand{\epsilon}{\varepsilon}
\renewcommand{\phi}{\varphi}
\newcommand{\N}{\mathbb{N}}

\newcommand{\Ccat}{\mathcal{C}}
\newcommand{\Dcat}{\mathcal{D}}
\newcommand{\Ecat}{\mathcal{E}}
\newcommand{\Hcat}{\mathcal{H}}
\newcommand{\Cbi}{\mathbf{C}}
\newcommand{\Dbi}{\mathbf{D}}

\newcommand{\B}{\mathcal{B}}

\newcommand{\op}{^\mathrm{op}}
\newcommand{\co}{^\mathrm{co}}

\newcommand{\inv}{^{-1}}
\newcommand{\id}{\mathrm{id}}
\newcommand{\Id}{\mathrm{Id}}

\newcommand{\Set}{\mathrm{Set}}
\newcommand{\Cat}{\mathbf{Cat}}
\newcommand{\Prof}{\mathbf{Prof}}
\newcommand{\Grpd}{\mathbf{Grpd}}
\newcommand{\Moncat}{\mathrm{Mon}}
\newcommand{\Monbi}{\mathbf{Mon}}
\newcommand{\Mod}{\mathbf{Biset}}
\newcommand{\Grpcat}{\mathrm{Grp}}

\newcommand{\Hom}{\mathrm{Hom}}
\newcommand{\End}{\mathrm{End}}
\newcommand{\Lan}{\mathrm{Lan}}
\newcommand{\LQProf}{\mathbf{LQProf}}
\newcommand{\LQMod}{\mathbf{LQBiset}}
\newcommand{\Lax}{\mathbf{Lax}}

\newcommand{\slashedrightarrow}{\mathrel{\ooalign{\hss$\vcenter{\hbox{\tikz{\draw (0,4.5pt) -- (0,0) ;}}}\mkern2.75mu$\hss\cr$\to$}}}

\newcommand{\splitext}[6]{%
\tikz[baseline]{
\newdimen{\mylabelwidth}
\newdimen{\skipwidth}
\node[anchor=base] (A) {\hspace*{\dimexpr0.5pt-\pgfkeysvalueof{/pgf/inner xsep}}${#1}$};
\settowidth{\mylabelwidth}{\pgfinterruptpicture {$#2$} \endpgfinterruptpicture}
\pgfmathsetlength{\skipwidth}{max(\mylabelwidth,10pt)}
;\node[right] (B) at ([xshift=\skipwidth+12pt]A.east) {${#3}$};
\settowidth{\mylabelwidth}{\pgfinterruptpicture {$#4$} \endpgfinterruptpicture}
\settowidth{\skipwidth}{\pgfinterruptpicture {$#5$} \endpgfinterruptpicture}
\pgfmathsetlength{\skipwidth}{max(\skipwidth,\mylabelwidth,10pt)}
\node[right] (C) at ([xshift=\skipwidth+12pt]B.east) {${#6}$\hspace*{\dimexpr0.5pt-\pgfkeysvalueof{/pgf/inner xsep}}};
\draw[normalTail->] (A) to node [above] {${#2}$} (B);
\draw[transform canvas={yshift=0.5ex},>=to,->>] (B) to node [above] {${#4}$} (C);
\draw[transform canvas={yshift=-0.5ex},-to] (C) to node [below] {${#5}$} (B);
}}

\newcommand{\normalext}[5]{%
\tikz[baseline]{
\newdimen{\mylabelwidth}
\newdimen{\skipwidth}
\node[anchor=base] (A) {\hspace*{\dimexpr0.5pt-\pgfkeysvalueof{/pgf/inner xsep}}${#1}$};
\settowidth{\mylabelwidth}{\pgfinterruptpicture {$#2$} \endpgfinterruptpicture}
\pgfmathsetlength{\skipwidth}{max(\mylabelwidth,12pt)}
\node[right] (B) at ([xshift=\skipwidth+10pt]A.east) {${#3}$};
\settowidth{\mylabelwidth}{\pgfinterruptpicture {$#4$} \endpgfinterruptpicture}
\pgfmathsetlength{\skipwidth}{max(\mylabelwidth,10pt)}
\node[right] (C) at ([xshift=\skipwidth+10pt]B.east) {${#5}$\hspace*{\dimexpr0.5pt-\pgfkeysvalueof{/pgf/inner xsep}}};
\draw[normalTail->] (A) to node [above] {${#2}$} (B);
\draw[>=to,->>] (B) to node [above] {${#4}$} (C);
}}

\newtcolorbox[crefname={Aside}{Asides},Crefname={Aside}{Asides},use counter*=theorem]{asidebox}[2][]{%
colback=white,colbacktitle=black!15,coltitle=black,
sharp corners=all,
breakable,parbox=false,#1,
title=\textbf{Aside \thetcbcounter.} \textnormal{#2}}

\title{Monoid extensions and the Grothendieck construction}

\author[G. Manuell]{Graham Manuell}
\address{CMUC, Department of Mathematics, University of Coimbra, Coimbra, Portugal}
\email{graham@manuell.me}
\thanks{The author acknowledges financial support from the Centre for Mathematics of the University of Coimbra (UIDB/00324/2020, funded by the Portuguese Government through FCT/MCTES)} %

\date{\today}


\begin{document}

\maketitle
\thispagestyle{empty}

\begin{abstract}
  In category theory circles it is well-known that the Schreier theory of group extensions can be understood in terms of the Grothendieck construction on indexed categories.
  However, it is seldom discussed how this relates to extensions of monoids.
  We provide an introduction to the generalised Grothendieck construction and apply it to recover classifications of certain classes of monoid extensions (including Schreier and weakly Schreier extensions in particular).
\end{abstract}

\setcounter{section}{-1}
\section{Introduction}

A split extension of groups \splitext{N}{}{G}{}{}{H} with kernel $N$ and cokernel $H$ can be uniquely specified by a group homomorphism from $H$ to the automorphism of $N$ by means of the semidirect product construction.
The classification of non-split extensions \normalext{N}{}{G}{}{H} is more subtle and involves a pair of maps $\phi\colon H \times N \to N$ and $\chi\colon H \times H \to N$ satisfying certain axioms, where $\phi$ is similar to, but not quite, an action of $N$ on $H$ and $\chi$ is a certain `correction' to $\phi$ generalising the elements of the second cohomology group from the case where $N$ is abelian.

As we explain below, this data can be understood in a more motivated and intuitive way in terms of the much more general \emph{Grothendieck construction} \cite{grothendieck1971categories} relating opfibrations over a category $\Hcat$ to pseudofunctors from $\Hcat$ into $\Cat$.
The connection between this construction and the theory of group extensions was noted by Grothendieck himself.
The potential of applying related constructions to monoid extensions is surely also understood by category theorists, but I have not seen it explicitly mentioned anywhere. (However, do see \cite{hoff1994cohomologies} which provides a categorical generalisation of one class of monoid extensions, but does not link this to the general theory.)
The aim of this paper is to make this relationship more widely known and to demonstrate the utility of this viewpoint.

The general theory of monoid extensions is not nearly as well behaved as that of group extensions and so it is standard to restrict to certain classes of monoid extensions (see \cite{faul2021survey} for a nice survey of some of these).
One such class is the \emph{Schreier extensions} of \cite{redei1952verallgemeinerung}.
\begin{definition}
 An extension of monoids \normalext{N}{k}{G}{e}{H} is \emph{Schreier} if for every $h \in H$ there is a $u_h \in e^{-1}(h)$ such that for each $g \in e^{-1}(h)$ there is a unique $n \in N$ with $g = k(n) u_h$.
\end{definition}
These are precisely the class of extensions whose cokernel map is a \emph{preopfibration} of the corresponding one-object categories and a modification of the Grothendieck construction may be applied. We will show how other modifications of the construction can be also used to describe other important classes of monoid extension, including \emph{weakly Schreier} extensions \cite{fleischer1981monoid,faul2021survey}.
Along the way we will also provide a characterisation of the morphisms of such extensions.

\section{Background}

We will assume the reader is familiar with the basics of category theory, which can be found in \cite{MacLane1998categories}.
In this section we will describe the more advanced notions that are necessary for understanding the paper.

We start with the notion of a bicategory. For a more thorough account of the theory of bicategories see \cite{johnson20212d} or \cite[Chapter 7]{borceux1994handbook1}. Recall that (small) categories and functors form a category $\Cat$, but there are also natural transformations between functors which act as `higher order morphisms'.
These 2-morphisms give $\Cat$ the structure of a 2-category. Bicategories are a generalisation of 2-categories that require composition of functors to only be associative and unital up to isomorphism instead of strictly so 
(just as how monoidal categories generalise strict monoidal categories).
These isomorphisms must then satisfy a number of coherence conditions as we explain below.
\begin{definition}
 A \emph{bicategory} $\Cbi$ consists of
 \begin{itemize}
  \item a class of \emph{objects} $\Cbi_0$,
  \item a category $\Cbi(X,Y)$ for each pair of objects $X,Y \in \Cbi_0$, the objects of which are called \emph{1-morphisms} (from $X$ to $Y$), the morphisms of which are called \emph{2-morphisms}, and where composition is called \emph{vertical composition} and denoted by juxtaposition,
  \item a 1-morphism $1_X \in \Cbi(X,X)$ for each object $X \in \Cbi_0$, called the \emph{identity 1-morphisms},
  \item a functor $c_{X,Y,Z}\colon \Cbi(Y,Z) \times \Cbi(X,Y) \to \Cbi(X,Z)$ for each triple $X,Y,Z \in \Cbi_0$, defining the \emph{horizontal composition} of 1-morphisms and 2-morphisms, which is denoted by juxtaposition or $\circ$ for 1-morphisms and by $\ast$ for 2-morphisms,
  \item a natural isomorphism $\lambda_{X,Y}\colon c_{X,Y,Y} \circ (1_Y \times \Id_{\Cbi(X,Y)}) \xrightarrow{\sim} \Id_{\Cbi(X,Y)}$ and a natural isomorphism $\rho_{X,Y}\colon c_{X,X,Y} \circ (\Id_{\Cbi(X,Y)} \times 1_X) \xrightarrow{\sim} \Id_{\Cbi(X,Y)}$ for each pair $X, Y \in \Cbi_0$, called the \emph{left} and \emph{right unitors} respectively,
  \item and a natural isomorphism $\alpha_{X,Y,Z,W}\colon c_{X,Y,W} \circ (c_{Y,Z,W} \times \Id_{\Cbi(X,Y)}) \xrightarrow{\sim} c_{X,Z,W} \circ (\Id_{\Cbi(Z,W)} \times c_{X,Y,Z})$ for each quadruple $X,Y,Z,W \in \Cbi_0$, called the \emph{associators},
 \end{itemize}
 subject to the conditions that
 \begin{itemize}
  \item the diagram
  \begin{center}
    \begin{tikzpicture}[auto]
     \node (A) {$(f \circ 1_Y) \circ g$};
     \node (B) [right=3cm of A] {$f \circ (1_Y \circ  g)$};
     \coordinate (mid) at ($(A)!0.5!(B)$);
     \node (C) [above=2.5cm of mid] {$f \circ g$};
     \draw[->] (A) to node {$\rho_f \ast \id_g$} (C);
     \draw[->] (B) to node [swap] {$\id_f \ast \lambda_g$} (C);
     \draw[->] (A) to node [swap] {$\alpha_{f,1_Y,g}$} (B);
    \end{tikzpicture}
  \end{center}
  commutes for each $f\colon Y \to Z$ and $g\colon X \to Y$ (where we have suppressed the subscripts for $\rho$, $\lambda$ and $\alpha$),
  \item and the diagram
  \begin{center}
    \begin{tikzpicture}[auto]
     \node (T) [regular polygon, regular polygon sides=5, minimum size=5.0cm, yscale=0.85] {};
     \node (A) at (T.corner 1) {$(fg)(hi)$};
     \node (B) at (T.corner 2) {$((fg)h)i$};
     \node (C) at (T.corner 3) {$(f(gh))i$};
     \node (D) at (T.corner 4) {$f((gh)i)$};
     \node (E) at (T.corner 5) {$f(g(hi))$};
     \draw[->] (B) to node {$\alpha_{fg,h,i}$} (A);
     \draw[->] (B) to node [swap, pos=0.4] {$\alpha_{f,g,h} \ast \id_i$} (C);
     \draw[->] (C) to node [swap] {$\alpha_{f,gh,i}$} (D);
     \draw[->] (D) to node [swap, pos=0.6] {$\id_f \ast \alpha_{g,h,i}$} (E);
     \draw[->] (A) to node {$\alpha_{f,g,hi}$} (E);
    \end{tikzpicture}
  \end{center}
  commutes for each $f\colon U \to V$, $g\colon Z \to U$, $h\colon Y \to Z$ and $i\colon X \to Y$.
 \end{itemize}
 We write $\Cbi\op$ for the bicategory $\Cbi$ with the 1-morphisms reversed and $\Cbi\co$ for the bicategory $\Cbi$ with the 2-morphisms reversed.
\end{definition}
Any category gives an example of a bicategory with only identity 2-morphisms.
As mentioned above, $\Cat$ itself is an example of a bicategory (where the associators and unitors are trivial).

Recall that a monoid $M$ can be viewed as a one-object category $\B M$ whose morphisms are elements of $M$ and where composition is given by monoid multiplication so that $m \circ n = mn$.
We will write $\star$ for the unique object of $\B M$.
Similarly, monoidal categories can be interpreted as one-object bicategories where the objects of the monoidal category correspond to the 1-morphisms of the bicategory.

Let us consider one-object categories in more detail.
A functor $F\colon \B L \to \B M$ is simply a monoid homomorphism from $L$ to $M$. A natural transformation from $F$ to another such functor $G$ is then given by an element of $m \in M$ such that $m F(\ell) = G(\ell)m$ for all $\ell \in L$. Vertical composition is given by multiplication and horizontal composition of $m \colon F \to G$ and $\ell \colon F' \to G'$ for $F',G' \colon \B K \to \B L$ and is given by $m \ast \ell = m F(\ell)$.
Thus, we have not arrived at the usual 1-category of monoids, but a bicategory $\Monbi$.

\begin{remark}\label{rem:monbi_vs_moncat}
In \cite[Section 5.6]{baez2010cohomology} it is explained that to get the usual category $\Moncat$ of monoids without nontrivial 2-morphisms, we should think of monoids as being one-object \emph{pointed} categories (in the sense of having a distinguished object; not as in being enriched over pointed sets).
Indeed, $\Moncat$ is equivalent to the coslice 2-category $1/\Monbi$.
There is exactly one map from $1$ to each monoid and these automatically commute with every monoid homomorphism. However, the compatibility on 2-morphisms ensures only trivial 2-morphisms between these.
\end{remark}

Another concept which will be important for us to understand is that of a profunctor between categories.
(For more details see \cite[Chapter 7]{borceux1994handbook1} or \cite{benabou2000distributors} where they are called \emph{distributors}.)
\begin{definition}
 A \emph{profunctor} $P\colon \Ccat \slashedrightarrow \Dcat$ is a functor $\Ccat\op \times \Dcat \to \Set$.
\end{definition}
In particular, every functor $F\colon \Ccat \to \Dcat$ gives a profunctor $H_F \colon \Ccat \slashedrightarrow \Dcat$ defined by $H_F(X,Y) = \Hom_\Dcat(F X, Y)$.

\begin{remark}
 Often the reversed convention is used, where a profunctor $\Ccat \slashedrightarrow \Dcat$ means a functor $\Dcat\op \times \Ccat \to \Set$.
 We have used the opposite convention so as to agree better with bisets and to minimise the use opposite categories in this paper.
\end{remark}

Recall that the functor category $\Set^{\B M} = \Cat(\B M, \Set)$ is isomorphic to the category of (left) $M$-sets.
A profunctor $\B N \slashedrightarrow \B M$ is then simply an $(M, N)$-biset --- that is, a set $X$ equipped with a left action by $M$ and a right action by $N$ such that $(m \cdot x) \cdot n = m \cdot (x \cdot n)$ for all $m \in M$, $n \in N$ and $x \in X$.

By the exponential adjunction in $\Cat$, a profunctor from $\Ccat$ to $\Dcat$ corresponds to a functor from $\Ccat\op$ to the functor category $\Set^\Dcat$. Moreover, the Yoneda embedding $Y\colon \Ccat\op \to \Set^\Ccat$ gives the free cocompletion of the (small) category $\Ccat\op$ and hence these functors in turn are in direct correspondence with \emph{cocontinuous} functors from $\Set^\Ccat$ to $\Set^\Dcat$.
Such cocontinuous functors can be composed in the obvious way leading to a bicategory $\Prof$ with (small) categories as objects, profunctors as 1-morphisms and natural transformations as 2-morphisms.

More directly, the profunctor composition of $P\colon \Dcat \slashedrightarrow \Ecat$ and $Q\colon \Ccat \slashedrightarrow \Dcat$ can be described by a coend
\[P \circ Q = \int^{D \in \Dcat} P(D,=) \times Q(-,D),\]
which can be understood as a categorification of the composition of relations.

Explicitly, the composition of profunctors $P\colon \Dcat \slashedrightarrow \Ecat$ and $Q\colon \Ccat \slashedrightarrow \Dcat$ is given by
\[(P \circ Q)(C,E) = \left(\bigsqcup_{D \in \Dcat} P(D,E) \times Q(C,D)\right)/{\sim}\] where ${\sim}$ is the equivalence relation generated by $(P(f,\id_E)(p),q)_{D} \sim (p,Q(\id_C,f)(q))_{D'}$ for $p \in P(D',E)$, $q \in Q(C,D)$ and $f \in \Dcat(D,D')$,
and \[(P \circ Q)(c,e)\colon [(p,q)_D] \mapsto [(P(\id_D,e)(p),Q(c,\id_D)(q))_D]\] for $c\colon C' \to C$, $e\colon E \to E'$, $p \in P(D,E)$ and $q \in Q(C,D)$.
The 2-morphisms between profunctors are simply given by natural transformations and the action of this composition on 2-morphisms can be defined by
\[(\phi \ast \kappa)_{C,E}\colon [(p,q)_D] \mapsto [(\phi_{D,E}(p),\kappa_{C,D}(q))_D]\]
for $\phi \colon P \to P'$ and $\kappa\colon Q \to Q'$.
Moreover, the identity profunctors are given by the $\Hom$ functors.
The unitors and associators can then also be defined in a straightforward manner (by the coend calculus or directly), but we omit the details.

The composition of profunctors between one-object categories corresponds to the \emph{tensor product} of the corresponding bisets.
Recall that if $X$ is an $(L,M)$-biset and $Y$ is an $(M,N)$-biset, then the tensor product $X \otimes Y$ is an $(L,N)$-biset with underlying set $\{(x,y) \in X \times Y\}/{\sim}$ where ${\sim}$ is the equivalence relation generated by $(x \cdot m, y) \sim (x, m \cdot y)$. The actions of $L$ and $N$ on $X \otimes Y$ are given by acting on the appropriate component of the pairs in the obvious way.
The identity profunctor on $\B N$ corresponds to the $(N,N)$-biset with underlying set $N$ and actions given by left and right multiplication in $N$.

We obtain a bicategory $\Mod$ whose objects are monoids, whose 1-morphisms are bisets, and whose 2-morphisms are biset homomorphisms where horizontal composition is given by tensor product of bisets.
Then $\Mod$ can be identified with the full sub-bicategory of $\Prof$ consisting of one-object categories.

Finally, we define the notion of lax functors between bicategories.
\begin{definition}
 Let $\Cbi$ and $\Dbi$ be bicategories. A \emph{lax functor} $L\colon \Cbi \to \Dbi$ consists of
 \begin{enumerate}[1)]
  \item a function $L\colon \Cbi_0 \to \Dbi_0$,
  \item a functor $L_{X,Y}\colon \Cbi(X,Y) \to \Dbi(L(X),L(Y))$ for each pair of objects $X, Y \in \Cbi_0$, which we usually also denote by $L$,
  \item a 2-morphism $\iota^L_X\colon 1_{L(X)} \to L(1_X)$ in $\Dbi$ for each object $X \in \Cbi_0$, called the \emph{unitors} of $L$,
  \item and a 2-morphism $\gamma^L_{f,g}\colon L(f) \circ L(g) \to L(f \circ g)$ in $\Dbi$ for each pair of composable 1-morphisms $(f,g)$ in $\Cbi$, called the \emph{compositors},
 \end{enumerate}
 subject to the conditions that
 \begin{enumerate}[a)]
  \item the diagram
  \begin{center}
    \begin{tikzpicture}[node distance=2.5cm, auto]
     \node (A) {$L(f)\circ L(g)$};
     \node (B) [right=1.5cm of A] {$L(fg)$};
     \node (C) [below of=A] {$L(f') \circ L(g')$};
     \node (D) [below of=B] {$L(f'g')$};
     \draw[->] (A) to node {$\gamma^L_{f,g}$} (B);
     \draw[->] (A) to node [swap] {$L(\sigma)\ast L(\tau)$} (C);
     \draw[->] (B) to node {$L(\sigma \ast \tau)$} (D);
     \draw[->] (C) to node [swap] {$\gamma^L_{f',g'}$} (D);
    \end{tikzpicture}
  \end{center}
  commutes for each pair of horizontally composable 2-morphisms $\sigma\colon f \to f'$ and $\tau\colon g \to g'$ in $\Cbi$,
  \item the two diagrams
  \begin{center}
   \begin{minipage}{.45\textwidth}
    \centering
    \begin{tikzpicture}[node distance=2.5cm, auto]
     \node (A) {$1_{L(Y)}\circ L(f)$};
     \node (B) [right=1.5cm of A] {$L(f)$};
     \node (C) [below of=A] {$L(1_Y) \circ L(f)$};
     \node (D) [below of=B] {$L(1_Y \circ f)$};
     \draw[->] (A) to node {$\lambda_{L(f)}$} (B);
     \draw[->] (A) to node [swap] {$\iota^L_Y \ast \id_{L(f)}$} (C);
     \draw[->] (C) to node [swap] {$\gamma^L_{1_Y,f}$} (D);
     \draw[->] (D) to node [swap] {$L(\lambda_f)$} (B);
    \end{tikzpicture}
   \end{minipage}
   \begin{minipage}{.45\textwidth}
    \centering
    \begin{tikzpicture}[node distance=2.5cm, auto]
     \node (A) {$L(f) \circ 1_{L(X)}$};
     \node (B) [right=1.5cm of A] {$L(f)$};
     \node (C) [below of=A] {$L(f) \circ L(1_X)$};
     \node (D) [below of=B] {$L(f \circ 1_X)$};
     \draw[->] (A) to node {$\rho_{L(f)}$} (B);
     \draw[->] (A) to node [swap] {$\id_{L(f)} \ast \iota^L_X$} (C);
     \draw[->] (C) to node [swap] {$\gamma^L_{f,1_X}$} (D);
     \draw[->] (D) to node [swap] {$L(\rho_f)$} (B);
    \end{tikzpicture}
   \end{minipage}
  \end{center}
  commute for each 1-morphism $f \in \Cbi(X,Y)$,
  \item and the diagram
  \begin{center}
    \begin{tikzpicture}[node distance=2.25cm, auto]
     \node (A) {$(L(f)\circ L(g)) \circ L(h)$};
     \node (B) [right=2.25cm of A] {$L(f)\circ (L(g) \circ L(h))$};
     \node (C) [below of=A] {$L(f \circ g) \circ L(h)$};
     \node (D) [below of=B] {$L(f) \circ L(g \circ h)$};
     \node (E) [below of=C] {$L((f\circ g) \circ h)$};
     \node (L) [below of=D] {$L(f \circ (g \circ h))$};
     \draw[->] (A) to node {$\alpha_{L(f),L(g),L(h)}$} (B);
     \draw[->] (A) to node [swap] {$\gamma^L_{f,g} \ast \id_{L(h)}$} (C);
     \draw[->] (B) to node {$\id_{L(f)} \ast \gamma^L_{g,h}$} (D);
     \draw[->] (C) to node [swap] {$\gamma^L_{fg,h}$} (E);
     \draw[->] (D) to node {$\gamma^L_{f,gh}$} (L);
     \draw[->] (E) to node {$L(\alpha_{f,g,h})$} (L);
    \end{tikzpicture}
  \end{center}
  commutes for all composable triples of 1-morphisms $(f,g,h)$ in $\Cbi$.
 \end{enumerate}
 We say a lax functor $L$ is \emph{unitary} if the unitors $\iota^L$ are invertible and we say $L$ is a \emph{pseudofunctor} if both the unitors and compositors are invertible.
 We say a lax functor is \emph{strictly unitary} or \emph{normal} if its unitors are identity 2-morphisms.
 
 It will be convenient to restrict to normal lax functors and normal pseudofunctors instead of arbitrary unitary ones. There is no loss of generality because every unitary lax functor (or pseudofunctor) is pseudonaturally isomorphic to a normal one.
\end{definition}
Pseudofunctors should be thought of as the correct bicategorical analogue of functors, while lax functors are a generalisation which does not appear in the setting of 1-categories.
Lax functors between one-object bicategories correspond to lax monoidal functors, while pseudofunctors correspond to strong monoidal functors.

The assignment of functors $F\colon \Ccat \to \Dcat$ to profunctors $H_F\colon \Ccat \slashedrightarrow \Dcat$ mentioned above can be extended to a pseudofunctor from $\Cat\co$ to $\Prof$ which is the identity on objects and fully faithful on hom-categories (by the Yoneda lemma).

Let us consider how this restricts to the case of monoids to give an inclusion from $\Monbi\co$ into $\Mod$.
A monoid homomorphism $f\colon N \to M$ is associated to a biset $H_f$ with underlying set $M$, left action given by multiplication $m' \cdot m = m'm$ and right action given by $m \cdot n= m f(n)$.
A 2-morphism in $\Monbi$ from $f$ to $f'\colon N \to M$ represented by $\mu \in M$ is sent to the biset homomorphism from $H_{f'}$ to $H_f$ mapping $m$ to $m \mu$.
The unitor at $M$ is given by the identity map on the biset $M$ (with both actions given by multiplication) and the compositors are the biset homomorphisms $H_f \otimes H_g \to H_{fg}$ sending $[(m,n)]$ to $m f(n)$.

Finally, just as there is a notion of natural transformation between functors, one can define \emph{oplax transformations} between lax functors and even \emph{modifications} between oplax transformations.
An oplax transformation $\tau$ from $L_1\colon \Cbi \to \Dbi$ to $L_2\colon \Cbi \to \Dbi$ consists of a 1-morphism $\tau_X\colon L_1(X) \to L_2(X)$ in $\Dbi$ for each object $X$ in $\Cbi$ and a 2-morphism $\tau_f\colon \tau_Y \circ L_1(f) \to L_2(f) \circ \tau_X$ in $\Dbi$ for each 1-morphism $f\colon X \to Y$ in $\Cbi$. These must satisfy some coherence conditions which we omit.
A modification $\aleph$ from $\tau\colon L_1 \to L_2$ to $\tau'\colon L_1 \to L_2$ consists of a 2-morphism $\aleph_X\colon \tau_X \to \tau'_X$ in $\Dbi$ for each object $X$ in $\Cbi$ and these must satisfy a certain naturality condition. (See \cite{johnson20212d} for details.)

\section{The Grothendieck--Bénabou correspondence}

The classical Grothendieck construction gives a correspondence between pseudofunctors from a (small) category $\Ccat$ (viewed as a locally discrete bicategory) to $\Cat$ and certain functors called \emph{opfibrations} with codomain $\Ccat$.
Let us start with a more general construction described in \cite{benabou2000distributors} that gives a correspondence between normal lax functors from $\Ccat$ to $\Prof$ and arbitrary functors into $\Ccat$.
In fact, it is more intuitive to consider the inverse construction.

Given a functor $F\colon \Ccat \to \Dcat$ between small categories, it is natural to consider its `fibres'. For every object $D \in \Dcat$ we define $\partial F(D)$
to be the subcategory of $\Ccat$ consisting of the objects which are mapped to $D$ by $F$ and the morphisms which are mapped to $\id_D$. (We consider \emph{strict} fibres to avoid needless complexity in our main case of interest.)

Now we might ask how these fibre categories vary along morphisms $f\colon A \to B$ in $\Dcat$. Here it is reasonable to consider the morphisms in $\Ccat$ which $F$ sends to $f$:
for every $\hat{A} \in \partial F(A)$ and every $\hat{B} \in \partial F(B)$ we set $\partial F(f)(\hat{A},\hat{B}) = \{ \hat{f} \in \Ccat(\hat{A},\hat{B}) \mid F(\hat{f}) = f \}$.
Indeed, $\partial F(f)$ can be made into a profunctor from $\partial F(A)$ to $\partial F(B)$ by inheriting its action on morphisms from the $\Hom$ functor in $\Ccat$.
Explicitly, for $a\colon \hat{A}' \to \hat{A}$ in $\partial F(A)$ and $b\colon \hat{B} \to \hat{B}'$ in $\partial F(B)$ the map $\partial F(f)(a,b)\colon F(f)(\hat{A},\hat{B}) \to F(f)(\hat{A}',\hat{B}')$ sends $\hat{f}$ to $b \hat{f} a$.

In this way we arrive at the following definition.
\begin{definition}\label{def:slice_to_lax_functor}
 Given a functor $F\colon \Ccat \to \Dcat$ we define a normal lax functor $\partial F\colon \Dcat \to \Prof$ by the following data.
 \begin{itemize}
  \item For each object $D \in \Dcat$, $\partial F(D)$ is the subcategory of $\Ccat$ with objects $C$ such that $F(C) = D$ and morphisms $c$ such that $F(c) = \id_D$.
  \item For each pair of objects $A, B \in \Dcat$ and each morphism $f\colon A \to B$, $\partial F(f)$ is a profunctor $\partial F(f) \colon \partial F(A) \slashedrightarrow \partial F(B)$
        defined by $\partial F(f)(\hat{A},\hat{B}) = \{ \hat{f} \in \Ccat(\hat{A},\hat{B}) \mid F(\hat{f}) = f \}$ on objects and $\partial F(f)(a,b)\colon \hat{f} \mapsto b \hat{f} a$ on morphisms.
        (Since $\Dcat$ has no nontrivial 2-morphisms, there is nothing to say about the action of $F$ on 2-morphisms.)
  \item For each object $D \in \Dcat$, note that $\partial F(\id_D) = \Hom_{\partial F(D)}$ and define the unitor $\iota^{\partial F}_D$ to be the identity natural transformation on $\Hom_{\partial F(D)}$.
  \item For each composable pair of morphisms $f\colon Y \to Z$ and $g\colon X \to Y$ in $\Dcat$, the compositor $\gamma^{\partial F}_{f,g}\colon \partial F(f) \circ \partial  F(g) \to \partial F(fg)$ is given by a natural transformation whose component at $(\hat{X},\hat{Z}) \in \partial F(X)\op \times \partial F(Z)$ sends $[(\hat{f},\hat{g})_{\hat{Y}}]$ to $\hat{f}\hat{g}$ (for $\hat{f}\colon \hat{Y} \to \hat{Z}$ and $\hat{g}\colon \hat{X} \to \hat{Y}$).
 \end{itemize}
\end{definition}

The (generalised) Grothendieck construction provides an inverse to this fibre construction.
Given a normal lax functor $L\colon \Dcat \to \Prof$ we define a category ${\int} L$ and a functor $\Pi_L \colon {\int} L \to \Dcat$.
\begin{definition}\label{def:lax_functor_to_slice}
 The category ${\int} L$ has objects of the form $(D,\hat{D})$ where $D$ is an object of $\Dcat$ and $\hat{D} \in L(D)$.
 Morphisms from $(D,\hat{D})$ to $(E,\hat{E})$ are given by pairs $(f,\hat{f})$ where $f\colon D \to E$ and $\hat{f} \in L(f)(\hat{D},\hat{E})$.
 The composite of morphisms $(f,\hat{f})\colon (D,\hat{D}) \to (E,\hat{E})$ and $(g,\hat{g})\colon (C,\hat{C}) \to (D,\hat{D})$ is given by $(fg, (\gamma^L_{f,g})_{\hat{C},\hat{E}}([(\hat{f},\hat{g})_{\hat{D}}]))$
 and the identity morphism on $(D,\hat{D})$ is given by $(\id_D,\id_{\smash{\hat{D}}})$, where we have used that $L(\id_D) = \Hom_{L(D)}$.
 The functor $\Pi_L\colon {\int} L \to \Dcat$ then simply projects out the first components of the objects and morphisms.
\end{definition}

These two constructions can be seen to be inverses up to isomorphism.
In fact, these form the object part of an equivalence between the slice 2-category $\Cat / \Dcat$
and the bicategory of normal lax functors from $\Dcat$ into $\Prof$, oplax transformations with 1-morphism components of the form $H_F$, and modifications.

Giving all the details of this equivalence is out of scope for this paper, but we give a brief description of the action of $\partial$ on 1-morphisms and 2-morphisms below. This is provided for motivation and completeness and the reader may feel free to skip this; we will give an independent proof of the restricted version we need in \cref{prop:correspondence_for_monoids}.

\begin{asidebox}[label=rem:full_benabou_equivalence]{The correspondence for morphisms}
We briefly describe one direction of the equivalence. We start by considering a 1-morphism in the (strict) slice 2-category $\Cat / \Dcat$. Suppose $F\colon \Ccat \to \Dcat$ and $F'\colon \Ccat' \to \Dcat$ are functors and let $\Phi\colon \Ccat \to \Ccat'$ be functor such that $F'\Phi = F$.
Then $\partial \Phi$ is an oplax transformation from $\partial F$ to $\partial F'$.
The 1-morphism component $(\partial \Phi)_D\colon \partial F (D) \slashedrightarrow \partial F' (D)$ is obtained from the appropriate restriction of the functor $\Phi$.

If $f\colon D \to E$ in $\Dcat$ then the 2-morphism component at $f$ is a natural transformation $(\partial \Phi)_f \colon (\partial \Phi)_E \circ \partial F (f) \to \partial F' (f) \circ (\partial \Phi)_D$. Since $(\partial \Phi)_D$ comes from a functor, the composite profunctor $\partial F' (f) \circ (\partial \Phi)_D$ has a particularly simple form (up to isomorphism) given by $(\hat{D}, \hat{E}') \mapsto \partial F'(f) (\Phi(\hat{D}), \hat{E}')$.

The other composite is trickier, but can be seen to correspond to a left Kan extension: $((\partial \Phi)_E \circ \partial F (f))(\hat{D}, -) = \Lan_{\Phi\vert_{\partial F (E)}} \partial F (f) (\hat{D},-)$. But now by the defining adjunction of the left Kan extension, maps from this into $\partial F'(f) (\Phi(=), -)$ are in bijection with maps from $\partial F (f)$ to $\partial F' (f) (\Phi(=), \Phi(-))$.

The required natural transformation $(\partial \Phi)_f$ is then induced straightforwardly from the map sending $\hat{f} \in \partial F (f) (\hat{D}, \hat{E})$ to $\Phi(\hat{f}) \in \partial F' (f) (\Phi(\hat{D}), \Phi(\hat{E}))$.
More explicitly, the $(\hat{D},\hat{E}')$ component of $(\partial \Phi)_f$ sends $[(h,\hat{f})_{\hat{E}}]$ to $[(h \circ \Phi(\hat{f}),\id_{\Phi(\hat{D})})_{\Phi(\hat{D})}]$.

Finally, let us consider a 2-morphism in $\Cat/ \Dcat$. Take $\Phi_{1,2}\colon \Ccat \to \Ccat'$ to be as above and suppose $\tau\colon \Phi_1 \to \Phi_2$ is a natural transformation such that $F'(\tau) = \id_F$. Then $\partial\tau$ is a modification from $\partial \Phi_1$ to $\partial \Phi_2$. Each component $(\partial \tau)_D$ is a natural transformation from the profunctor $(\partial \Phi_1)_D$ to $(\partial \Phi_2)_D$ induced by the appropriate restriction of $\tau$.
\end{asidebox}

An important consequence of this correspondence can be recovered by restricting to the case of normal lax functors $L\colon \Dcat \to \Prof$ that factor through the inclusion $H\colon \Cat\co \hookrightarrow \Prof$.
We consider for which functors $F\colon \Ccat \to \Dcat$ we have that $\partial F$ is of this form. This means that for each morphism $f\colon A \to B$, the profunctor $\partial F(f)$ is isomorphic to one of the form $\Hom(f_*(-),=)$ for some functor $f_*\colon \partial F(A) \to \partial F(B)$ (characterised up to isomorphism).
So the set $\partial F(f)(\hat{A},\hat{B})$ of morphisms $\hat{f}\colon \hat{A} \to \hat{B}$ mapped to $f$ under $F$ is (naturally) in bijection with the morphisms from some object $f_*(\hat{A})$ to $\hat{B}$ in $\partial F(B)$.
The identity map from $f_*(\hat{A})$ to $f_*(\hat{A})$ then corresponds to a certain universal map from $\hat{A}$ to $f_*(\hat{A})$.
Thus, we are lead to the following definition.
\begin{definition}
 Consider a functor $F\colon \Ccat \to \Dcat$. We say a morphism $u\colon \hat{A} \to \hat{B}$ in $\Ccat$ is \emph{pre-opcartesian} with respect to $F$ if for every $\hat{g}\colon \hat{A} \to \hat{B}'$ in $\Ccat$ with $F(\hat{g}) = F(u)$, there exists a unique morphism $b\colon \hat{B} \to \hat{B}'$ in $\partial F(F\hat{B})$ such that $bu = \hat{g}$.
 The functor $F$ is called a \emph{preopfibration} if for every morphism $f\colon F(\hat{A}) \to B$ there exists a pre-opcartesian lifting $u_f\colon \hat{A} \to \hat{B}$ in $\Ccat$ such that $F(u_f) = f$.
 If moreover, the class of pre-opcartesian morphisms is closed under composition, then we say $F$ is an \emph{opfibration}.
\end{definition}
The above correspondence then reduces to a correspondence between preopfibrations over $\Dcat$ and normal lax functors from $\Dcat$ into $\Cat\co$. (These are the same as \emph{oplax} functors from $\Dcat$ to $\Cat$, but we will stick to using lax functors to make the link to the general case clearer.)
The classical Grothendieck construction concerned a further restriction to a correspondence between opfibrations over $\Dcat$ and pseudofunctors from $\Dcat$ to $\Cat$.

\section{Schreier extensions}\label{sec:schreier_extensions}

We are now in a position to discuss the most immediate application to the theory of monoid extensions.
To do this we restrict all the categories involved to have precisely one object.
Recall that an extension of monoids \normalext{N}{k}{G}{e}{H} is \emph{Schreier} if for every $h \in H$ there is a $u_h \in e^{-1}(h)$ such that for each $g \in e^{-1}(h)$ there is a unique $n \in N$ with $k(n) u_h = g$.
The cokernel map in such an extension is called a \emph{Schreier epimorphism}.
\begin{lemma}
 A monoid homomorphism $e\colon G \to H$ is a Schreier epimorphism if and only if $\B e$ is a preopfibration.
\end{lemma}
\begin{proof}
 The Schreier condition is precisely the condition that every morphism has a pre-opcartesian lifting restricted to the case of monoids. Furthermore, the Schreier condition alone entails that $e$ is surjective (and indeed, the cokernel of its kernel).
\end{proof}

Now applying the correspondence between preopfibrations and normal lax functors into $\Cat\co$, we have that Schreier extensions with codomain $H$ correspond to certain normal lax functors from $\B H$ to $\Cat\co$.
In particular, in this case the fibre also has only one object and the normal lax functors are precisely those that factor through $\Monbi\co \hookrightarrow \Cat\co$.

Let us see what the data of a normal lax functor gives in this case. We have following (looking at the definition of a lax functor point by point).
\begin{enumerate}[1)]
 \item An object of $\Monbi$ for the single object of $\B H$ --- that is, a monoid $N$. This corresponds to the kernel of the extension, since it will be the fibre of the identity.
 \item A function from $H$ to the objects of $\Monbi(N,N)$ (since $\B H$ has one object and only trivial 2-morphisms). We can express this function in uncurried form as a map $\phi\colon H \times N \to N$
       such that $\phi(h,1) = 1$ and $\phi(h,nn') = \phi(h,n)\phi(h,n')$.
 \item A requirement that $\phi(1,-)$ is the identity --- that is, $\phi(1,n) = n$.
 \item A 2-morphism in $\Monbi\co$ for each $h,h' \in H$ from $\phi(h,-) \circ \phi(h',-)$ to $\phi(hh',-)$. Explicitly, this is an element $\chi(h,h') \in N$ such that $\phi(h,\phi(h',n))\chi(h,h') = \chi(h,h')\phi(hh',n)$.
\end{enumerate}
These must then additionally satisfy the following conditions.
\begin{enumerate}[a)]
 \item The first condition is trivial, since $\B H$ has no nontrivial 2-morphisms.
 \item The next gives $\chi(1,h) = 1$ and $\chi(h,1) = 1$.
 \item Finally, interpreting the compositor condition $\gamma^L_{xy,z} ( \gamma^L_{x,y} \ast \id_{L(z)} ) = \gamma^L_{x,yz} ( \id_{L(x)} \ast \gamma^L_{y,z} )$ in $\Monbi\co$ in terms of our data, we need $\chi(x,y)\chi(xy,z) = \phi(x,\chi(y,z))\chi(x,yz)$.
\end{enumerate}
This is precisely the data Rédei used (up to variance) to characterise Schreier extensions \cite{redei1952verallgemeinerung,faul2021survey}.

Do note, however, that the correspondence between extensions and these normal lax functors is only an equivalence, not an isomorphism. Thus, the precise correspondence is between \emph{isomorphism classes} of extensions and an appropriate corresponding equivalence relation on the above data. This requires looking at the morphisms, which will do in the next section.

\begin{remark}
 We can easily restrict the above characterisation to describe important special cases of Schreier extensions. A preopfibration $F\colon \Ccat \to \Dcat$ is called \emph{groupoidal} if \emph{every} morphism in $\Ccat$ is (pre)opcartesian with respect to $F$. Any such preopfibration is automatically an opfibration and they correspond to pseudofunctors into the 2-category $\Grpd$ of groupoids.
 In the case of monoids, this condition gives the so-called \emph{special Schreier extensions} of monoids: extensions \normalext{N}{k}{G}{e}{H} such that whenever $e(g) = e(g')$ there is a unique $n \in N$ such that $g = k(n) g'$ (see \cite{martins2016baer,faul2021survey}).
 Our characterisation is then as above but where $N$ is required to be a group.
 Finally, if $N$ further required to be abelian, this reduces to the cohomological characterisation given in \cite{martins2016baer} (once we take into account the equivalence relation defined in \cref{sec:morphisms_of_extensions}).
\end{remark}

Of course, we can now explicitly construct an extension from the data above using the (generalised) Grothendieck construction.
Let $L\colon \B H \to \Cat\co$ be the lax functor defined by a pair $(\phi, \chi)$ as above. The resulting category ${\int} L$ will have a single object and morphisms given by pairs $(h,n)$
where $h \in H$ and $n \in \Hom_{\B N}(L(h)(\star),\star) = N$.
These morphisms are the elements of the middle monoid in the extension. Multiplication is given by
\begin{align*}
 (h,n) \cdot (h',n') &= (hh', (\gamma^{H_L}_{h,h'})_{\star,\star}([(n,n')_\star])) \\
                     &= (hh', (\gamma^{H_L}_{h,h'})_{\star,\star}([(n\phi(h,n'),1)_\star])) \\
                     &= (hh', n\phi(h,n')\chi(h,h')).
\end{align*}
The cokernel map is then the projection $(h,n) \mapsto h$ and the kernel is given by $n \mapsto (1,n)$.
Again, this recovers the usual construction of an extension from the data $(\phi, \chi)$.

\section{Morphisms of extensions}\label{sec:morphisms_of_extensions}

The full equivalence of bicategories mentioned in \cref{rem:full_benabou_equivalence} restricts to a correspondence between the full sub-2-category of $\Monbi/H$ consisting of the Schreier epimorphisms and a particular sub-2-category of $\Lax(\B H,\Monbi\co)$.

However, morphisms of extensions should not just commute with the quotient maps, but also the kernel maps.
If $e\colon G \to H$ and $e'\colon G' \to H$ are Schreier epimorphisms of monoids and $t\colon G \to G'$ is a monoid homomorphism satisfying $e't = e$, then by the correspondence we have an oplax transformation $\partial t\colon \partial e \to \partial e'$. But now note $N = (\partial e)(\star)$ and $N' = (\partial e')(\star)$ are the kernels of $e$ and $e'$ respectively, and so the 1-morphism component $(\partial t)_\star\colon N \to N'$ is the restriction of $t$ to the kernels. Thus, $t$ gives a morphism of extensions with cokernel $H$ and kernel $N$ precisely when the 1-morphism component of $\partial t$ is the identity on $N$. Such an oplax transformation is (essentially) what is known as an \emph{icon}.

The restriction to maps which fix the kernel also has the welcome side effect of killing the nontrivial 2-morphisms from $\Monbi$ similarly to as in \cref{rem:monbi_vs_moncat}.
Thus, here isomorphism/equivalence classes of extensions are precisely what we expect them to be.

Finally, the restriction to Schreier epimorphisms with fixed kernel $N$ allows us to rephrase the equivalence in terms of \emph{monoidal categories}, which might be more familiar to some readers.
This is because requiring the kernel of the Schreier epimorphism be $N$ amounts to asking that the lax functor $\partial e$ sends the unique object in $\B H$ to the object $N$ in $\Monbi\co$.
In other words, $\partial e$ factors through the one-object sub-bicategory of $\Monbi\co$ at $N$, which we could write as $\B (\End_{\Monbi\co}(N))$.

Lax functors between one-object bicategories correspond to lax monoidal functors between their endomorphism bicategories.
Moreover, icons between such lax functors correspond to \emph{monoidal natural transformations} between these lax monoidal functors, whose definition we now recall.
\begin{definition}
 Let $L_1, L_2\colon \Ccat \to \Dcat$ be lax monoidal functors. A \emph{monoidal natural transformation} $\tau\colon L_1 \to L_2$ is a natural transformation between the underlying functors such that the following diagrams commute (where $I$ represents the monoidal unit).
 \begin{center}
  \begin{minipage}{.35\textwidth}
  \centering
  \begin{tikzpicture}[node distance=2.5cm, auto]
    \node (A) {$I$};
    \node (B) [right of=A] {$L_1(I)$};
    \node (D) [below of=A] {$L_2(I)$};
    \draw[->] (A) to node {$\iota^{L_1}$} (B);
    \draw[->] (A) to node [swap] {$\iota^{L_2}$} (D);
    \draw[->] (B) to node {$\tau_I$} (D);
  \end{tikzpicture}
  \end{minipage}
  \begin{minipage}{.50\textwidth}
  \centering
  \begin{tikzpicture}[node distance=2.5cm, auto]
    \node (A) {$L_1(X) \otimes L_1(Y)$};
    \node (B) [right=1.5cm of A] {$L_1(X \otimes Y)$};
    \node (C) [below of=A] {$L_2(X) \otimes L_2(Y)$};
    \node (D) [below of=B] {$L_2(X \otimes Y)$};
    \draw[->] (A) to node {$\gamma^{L_1}_{X,Y}$} (B);
    \draw[->] (A) to node [swap] {$\tau_X \otimes \tau_Y$} (C);
    \draw[->] (B) to node {$\tau_{X \otimes Y}$} (D);
    \draw[->] (C) to node [swap] {$\gamma^{L_2}_{X,Y}$} (D);
  \end{tikzpicture}
  \end{minipage}
 \end{center}
\end{definition}

We have arrived at the following proposition (which will also follow from \cref{prop:correspondence_for_monoids} below).
\begin{proposition}
 There is an equivalence between the category of Schreier extensions of monoids with kernel $N$ and cokernel $H$ and the category of normal lax monoidal functors from $H$ (viewed as a discrete monoidal category) to $\End_{\Monbi\co}(N)$ and monoidal natural transformations.
\end{proposition}

We can now use this to complete the characterisation of Schreier extensions from \cref{sec:schreier_extensions} (while also giving a characterisation of morphisms of Schreier extensions).
We have argued that morphisms of Schreier extensions correspond to monoidal natural transformations.
In terms of our data, a morphism from $(\phi_1,\chi_2)$ to $(\phi_2,\chi_2)$ is given by
\begin{itemize}
 \item a 2-morphism in $\Monbi\co$ from $\phi_1(h,-)$ to $\phi_2(h,-)$ for each $h \in H$. We can specify this by a function $\psi\colon H \to N$ such that $\phi_1(h,n)\psi(h) = \psi(h)\phi_2(h,n)$.
\end{itemize}
This is subject to the following conditions.
\begin{itemize}
 \item The naturality condition is trivial since $H$ is a discrete monoidal category.
 \item The unit law says $\psi(1) = 1$,
 \item The tensor product condition means $\chi_1(h,h') \psi(hh') = \phi_1(h, \psi(h')) \psi(h) \chi_2(h,h')$ --- or equivalently using the equality above, $\chi_1(h,h') \psi(hh') = \psi(h) \phi_2(h, \psi(h')) \chi_2(h,h')$.
\end{itemize}

Thus, two pairs $(\phi_1, \chi_1)$ and $(\phi_2, \chi_2)$ give isomorphic extensions if and only if there is a function $\psi\colon H \to N$ satisfying the conditions above and where $\psi(h)$ is invertible for each $h \in H$. This can be to seen to agree with the equivalence relation described in \cite{faul2021survey}
and completes the characterisation.

Above we have restricted our discussion of morphisms to the case of Schreier extensions, but the argument works equally well for arbitrary extensions of monoids.
We simply use $\Mod$ in place of $\Monbi\co$. This gives us the following result, which we will prove from first principles for completeness and so as to avoid the more complicated prerequisites.

\begin{proposition}\label{prop:correspondence_for_monoids}
 Fix two monoids $N$ and $H$. Consider the category of exact sequences of monoids $0 \to N \to G \to H$ --- in other words, the category
 whose objects are triples ($G$,$e$,$k$) where $G$ is a monoid, $e\colon G \to H$ is an \emph{arbitrary} monoid homomorphism and $k$ is the kernel of $e$,
 and whose morphisms are maps of the form $t\colon G \to G'$ that make the obvious triangles commute.
 There is an equivalence between this category and the category of normal lax monoidal functors from $H$ to $\End_\Mod(N)$ and monoidal natural transformations between them.
\end{proposition}
\begin{proof}
 The correspondence on the level of objects is obtained from \cref{def:slice_to_lax_functor,def:lax_functor_to_slice} as explained above.
 We now describe the action of these functors on morphisms.
 
 Consider the following morphism of exact sequences.
 \begin{center}
   \begin{tikzpicture}[node distance=2.0cm, auto]
    \node (A) {$N$};
    \node (B) [right of=A] {$G$};
    \node (C) [right of=B] {$H$};
    \node (D) [below of=A] {$N$};
    \node (E) [right of=D] {$G'$};
    \node (F) [right of=E] {$H$};
    \draw[normalTail->] (A) to node {$k$} (B);
    \draw[->] (B) to node {$e$} (C);
    \draw[normalTail->] (D) to [swap] node {$k'$} (E);
    \draw[->] (E) to [swap] node {$e'$} (F);
    \draw[->] (B) to node {$t$} (E);
    \draw[double equal sign distance] (A) to (D);
    \draw[double equal sign distance] (C) to (F);
   \end{tikzpicture}
 \end{center}
 The monoidal transformation $\partial t$ between the lax monoidal functors $\partial e, \partial e' \colon H \to \End_\Mod(N)$
 is defined by setting $(\partial t)_h (x) = t(x)$.
 
 This is well-defined since $x \in \partial e (h)$ means that $x \in G$ with $e(x) = h$,
 and then $t(x) \in G'$ with $e'(t(x)) = e(x) = h$ by the commutativity of the above diagram, so that $t(x) \in \partial e' (h)$.
 
 We now show it is a monoidal natural transformation. Naturality is immediate since $H$ is a discrete monoidal category.
 The unit condition for monoidal natural transformations requires that $(\partial t)_1$ be the identity on $N$, which follows from the commutativity of the diagram above.
 For the tensor product condition we need that the following diagram commutes.
 \begin{center}
  \begin{tikzpicture}[node distance=2.5cm, auto]
    \node (A) {$\partial e(h_1) \otimes \partial e(h_2)$};
    \node (B) [right=1.5cm of A] {$\partial e(h_1h_2)$};
    \node (C) [below of=A] {$\partial e'(h_1) \otimes \partial e'(h_2)$};
    \node (D) [below of=B] {$\partial e'(h_1h_2)$};
    \draw[->] (A) to node {$\gamma^{\partial e}_{h_1,h_2}$} (B);
    \draw[->] (A) to node [swap] {$(\partial t)_{h_1} \otimes (\partial t)_{h_2}$} (C);
    \draw[->] (B) to node {$(\partial t)_{h_1h_2}$} (D);
    \draw[->] (C) to node [swap] {$\gamma^{\partial e'}_{h_1,h_2}$} (D);
  \end{tikzpicture}
 \end{center}
 To see this consider $[(x,y)] \in \partial e(h_1) \otimes \partial e(h_2)$, where $x \in e\inv(h_1)$ and $y \in e\inv(h_2)$.
 Then along the upper path we have $[(x,y)] \mapsto xy \mapsto t(xy)$, while on the lower path $[(x,y)] \mapsto [(t(y),t(x))] \mapsto t(x)t(y)$.
 These agree since $t$ is a monoid homomorphism.
 
 For the other direction let $L, L'\colon H \to \End_\Mod(N)$ be normal lax monoidal functors and consider a monoidal natural transformation $\tau\colon L \to L'$.
 We define a morphism ${\int}\tau$ between the exact sequences $0 \to N \to {\int}L \to H$ and $0 \to N \to {\int}L' \to H$ by ${\int}\tau (h,\hat{h}) = (h, \tau_h(\hat{h}))$.
 
 To see that ${\int}\tau$ is a monoid homomorphism first note that ${\int}\tau (1,1) = (1, \tau_1(1)) = (1,1)$ by the unit condition on the monoidal natural transformation $\tau$. Now observe the following sequence of equalities.
 \begin{align*}
  {\textstyle\int}\tau ( (h_1,\hat{h}_1) \cdot (h_2,\hat{h}_2) ) &= {\textstyle\int}\tau ( (h_1 h_2, \gamma^L_{h_1,h_2}(\hat{h}_1,\hat{h}_2)) ) \\
                                                       &= (h_1 h_2, \tau_{h_1h_2}\gamma^L_{h_1,h_2}(\hat{h}_1,\hat{h}_2)) \\
                                                       &= (h_1 h_2, \gamma^{L'}_{h_1,h_2}(\tau_{h_1}(\hat{h}_1),\tau_{h_2}(\hat{h}_2))) \\
                                                       &= (h_1, \tau_{h_1}(\hat{h}_1)) \cdot (h_2, \tau_{h_2}(\hat{h}_2)) \\
                                                       &= {\textstyle\int}\tau (h_1, \hat{h}_1) \cdot {\textstyle\int}\tau (h_2, \hat{h}_2)
 \end{align*}
 Here the third equality follows from the tensor product condition on $\tau$.
 
 Furthermore, it is clear that ${\int}\tau$ commutes with the cokernel maps ${\int}L \to H$ and ${\int}L' \to H$.
 The fact that it commutes with the kernel inclusions then follows from the unit condition on $\tau$.
 
 It is easy to see both the maps $\partial$ and ${\int}$ preserve composition and hence define functors.
 It remains to show that they are inverses up to natural isomorphism, but this is routine and we omit the details.
\end{proof}

\begin{example}
 Consider the monoid homomorphism $e\colon \N \times \N \to \N$ given by addition. This has trivial kernel since $x + y = 0$ if and only if $x = y = 0$.
 By the above this corresponds to a normal lax monoidal functor from $\N$ to $\End_\Mod(1) \cong (\Set, \times)$. In particular, we have an $\N$-indexed family of sets corresponding to the different fibres of $e$.
 Here the $i$th set $S_i$ is isomorphic to $\{n \in \N \mid n \le i\}$. Then the compositors $\gamma_{i,j}\colon S_i \times S_j \to S_{i+j}$ give the multiplication operation between the fibres and are defined by $\gamma_{i,j}(n,m) = n+m$.
\end{example}

\begin{remark}
 We can also use the general equivalence of bicategories of \cref{rem:full_benabou_equivalence} to obtain other related results with very little additional work.
 In particular, \emph{split extensions} of monoids can be understood by first considering `points' of the bicategories involved (that is, the maps from the terminal object) before fixing the kernel.
 For instance, starting with the correspondence between opfibrations and pseudofunctors restricted to the case of groups this gives split epimorphisms of groups into $H$ on the one side and (1-)functors from $\B H$ into $\Grpcat$ on the other.
 Here the equivalence gives rise the usual semidirect product construction.
 Applying a similar idea in the case of Schreier extensions one can recover the known characterisation of Schreier \emph{split} extensions (see \cite{martins2013semidirect}),
 as long as we first restrict the 1-morphisms in the bicategories so that on the lax functor side we only have pseudonatural transformations (instead of general oplax transformations).
\end{remark}

\section{Weakly Schreier extensions}

In theory, it should be possible to use this approach to characterise any desired class of monoid extensions.
As a slightly more complicated example, we will now apply it to the case of \emph{weakly} Schreier extensions of monoids.

\begin{definition}
An extension of monoids \normalext{N}{k}{G}{e}{H} is \emph{weakly Schreier} if for every $h \in H$ there is a $u_h \in e^{-1}(h)$ such that for each $g \in e^{-1}(h)$ there is a (not necessarily unique) $n \in N$ with $k(n) u_h = g$.
\end{definition}

If Schreier extensions correspond to preopfibrations, weakly Schreier extensions should correspond to a weakened notion of preopfibration where uniqueness is removed from the definition of pre-opcartesian.

\begin{definition}
 Consider a functor $F\colon \Ccat \to \Dcat$. We say a morphism $u\colon \hat{A} \to \hat{B}$ in $\Ccat$ is \emph{weakly pre-opcartesian} with respect to $F$ if for every $\hat{g}\colon \hat{A} \to \hat{B}'$ in $\Ccat$ with $F(\hat{g}) = F(u)$, there exists a morphism $b\colon \hat{B} \to \hat{B}'$ in $\partial F(F\hat{B})$ such that $bu = \hat{g}$.
 The functor $F$ is called a \emph{weak preopfibration} if for every morphism $f\colon F(\hat{A}) \to B$ there exists a weakly pre-opcartesian lifting $u_f\colon \hat{A} \to \hat{B}$ in $\Ccat$ such that $F(u_f) = f$.
\end{definition}

We now consider the resulting conditions satisfied by the corresponding normal lax functor $\partial F\colon \Dcat \to \Prof$.

Let $f\colon A \to B$ be a morphism in $\Dcat$. Then for $\hat{A}$ lying over $A$ and $\hat{B}$ lying over $B$ recall that we have $\partial F (f)(\hat{A},\hat{B}) = \{\hat{f} \in \Ccat(\hat{A},\hat{B}) \mid F(\hat{f}) = f\}$.
By the weak preopfibration condition we have a weakly pre-opcartesian lift $u_f\colon \hat{A} \to \widetilde{f}(\hat{A})$ of $f$.
Now composition with $u_f$ defines a map sending $b \in \Hom_{\partial F (B)}(\widetilde{f}(\hat{A}), \hat{B})$ to $b u_f \in \partial F (f)(\hat{A},\hat{B})$. That $u_f$ is weakly pre-opcartesian is precisely the claim that this map is surjective.
Moreover, it is easy to see this map is natural in $\hat{B}$.
Thus, the functor $\partial F (f) (\hat{A},-)$ is a \emph{quotient} of the representable functor $\Hom(\widetilde{f}(\hat{A}),-)$ in the category $\Set^{\partial F(B)}$.

This suggests restricting to profunctors $P\colon \Ccat \slashedrightarrow \Dcat$ whose corresponding cocontinuous functors $\overline{P}\colon \Set^\Ccat \to \Set^\Dcat$
send representable functors to quotients of representable functors.
Note that such cocontinuous functors are closed under composition since cocontinuous functors preserve quotients and quotient maps are closed under composition. Thus, these profunctors define a wide, locally full sub-bicategory of $\Prof$ (that is, a sub-bicategory with a restricted set of 1-morphisms, but with the same objects and the same 2-morphisms between each pair of chosen 1-morphisms as the parent category).
Let us call this bicategory $\LQProf$.

\begin{proposition}\label{prop:weak_preopfibration_correspondence}
 The Grothendieck--Bénabou correspondence between functors into $\Dcat$ and normal lax functors from $\Dcat$ to $\Prof$ restricts to a correspondence between weak preopfibrations into $\Dcat$ and normal lax functors from $\Dcat$ to $\LQProf$.
\end{proposition}
\begin{proof}
 We have shown above that weak preopfibrations are sent to normal lax functors that factor through $\LQProf \hookrightarrow \Prof$. We now show that the functor obtained by applying the Grothendieck construction to a normal lax functor $L\colon \Dcat \to \LQProf$ is a weak preopfibration.
 
 Recall that the functor $\Pi_L\colon {\int}L \to \Dcat$ sends $(D,\hat{D}) \in {\int}L$ to $D \in \Dcat$.
 Consider $f\colon A \to B$ in $\Dcat$. Then by assumption, for each object $\hat{A} \in L(A)$, the functor $L(f)(\hat{A},-)$ is a quotient of some representable functor $\Hom_{L(B)}(\widetilde{f}(\hat{A}),-)$.
 By the Yoneda lemma, this quotient map sends each map $b\colon \widetilde{f}(\hat{A}) \to \hat{B}$ in $L(B)$ to $L(f)(\hat{A},b)(u_f)$ for some element $u_f \in L(f)(\hat{A},\widetilde{f}(\hat{A}))$.
 
 We claim the map $(f,u_f)$ in ${\int} L$ is a pre-opcartesian lifting of $f$ with respect to $\Pi_L$. We must show that every lift $(f,\hat{f})\colon (A, \hat{A}) \to (B, \hat{B})$ of $f$ can be expressed as the composite $(\id_B, b) \circ (f,u_f)$ for some morphism $b\colon \widetilde{f}(\hat{A}) \to \hat{B}$ in $L(B)$.
 
 By the definition of composition in ${\int} L$ we have $(\id_B, b) \circ (f,u_f) = (f, (\gamma^L_{\id_B, f})_{\hat{A},\hat{B}}([(b,u_f)_{\widetilde{f}(\hat{A})}]))$.
 Then by the form of the equivalence relation defined for the composition of profunctors we have $[(b,u_f)_{\widetilde{f}(\hat{A})}] = [(L(\id_B)(b,\id_{\hat{B}})(\id_{\hat{B}}),u_f)_{\widetilde{f}(\hat{A})}] = [(\id_{\hat{B}},L(f)(\id_{\hat{A}},b)(u_f))_{\hat{B}}]$.
 But now by the unitor commutative diagrams for $L$ and normality we have $\gamma^L_{\id_B, f} = \lambda_{L(f)}$ and hence $(\id_B, b) \circ (f,u_f) = (f,\lambda_{L(f)}([(\id_{\hat{B}},L(f)(\id_{\hat{A}},b)(u_f))_{\hat{B}}])) = (f, L(f)(\id_{\hat{A}},b)(u_f))$. So our claim will follow once we prove there is a morphism $b \in L(B)$ such that $\hat{f} = L(f)(\id_{\hat{A}},b)(u_f)$,
 but this is precisely what it means for the map from $\Hom_{L(B)}(\widetilde{f}(\hat{A}),\hat{B})$ to $L(f)(\hat{A},\hat{B})$ to be surjective and thus we are done.
\end{proof}

We will now consider what this correspondence means for monoid extensions.
Note that if $M$ is a monoid, the only representable functor in $\Set^{\B M}$ is given by the canonical left $M$-set structure on $M$ given by multiplication.
So the restriction of $\LQProf$ to the case of monoids gives the sub-bicategory of $\Mod$ whose 1-morphisms are the $(M,N)$-bisets which as left $M$-sets are quotients of the canonical left $M$-set structure on $M$ (with arbitrary compatible right $N$-set structure). Let us call this category $\LQMod$.
Now as a corollary of \cref{prop:weak_preopfibration_correspondence,prop:correspondence_for_monoids} we obtain the following equivalence.
\begin{proposition}
 Let $N$ and $H$ be monoids. There is an equivalence of categories between the category of weakly Schreier extensions of monoids with kernel $N$ and cokernel $H$
 and the category of normal lax monoidal functors from $H$ to $\End_\LQMod(N)$ and monoidal natural transformations between them.
\end{proposition}

We wish to see what data is needed to specify such a normal lax monoidal functor.
We must first look at the bicategory $\LQMod$ in more detail.
The discussion proceeds similarly to the proof that $\Monbi\co$ embeds into $\Mod$.

Explicitly, a morphism from $N$ to $M$ in $\LQMod$ can be described as follows. The underlying set and left $M$-set structure is given by a quotient $M/{\sim}$ where the equivalence relation ${\sim}$ is a left $M$-set congruence.
Then since $[m] \cdot n = (m \cdot [1]) \cdot n = m \cdot ([1] \cdot n)$, the right action of $N$ on $M/{\sim}$ is then specified by where it sends $[1]$.
Thus, the action can be defined by a function $f\colon N \to M$ by setting $[m] \cdot n = [mf(n)]$. Of course, two such functions $f_1, f_2$ will give the same right $N$-set structure whenever $f_1(n) \sim f_2(n)$ for all $n \in N$.
Finally, the equalities $[f(1)] = [1] \cdot 1 = [1]$ and $[f(nn')] = [1] \cdot nn' = ([1] \cdot n) \cdot n' = [f(n)] \cdot n' = [f(n)f(n')]$ show that such a map $f$ should preserve the unit and multiplication up to equivalence, and that is indeed enough to ensure this defines a valid $(M,N)$-biset structure.
We might call data such as $({\sim},f)$ a \emph{relaxed} monoid homomorphism, mimicking the `relaxed action' terminology of \cite{faul2021survey}.

Now consider the tensor product $X \otimes Y$ of an $(L,M)$-biset $X$ and an $(M,N)$-biset $Y$ in $\LQMod$ defined by ${\sim}^X$ and $f^X\colon M \to L$ and ${\sim}^Y$ and $f^Y\colon N \to M$ respectively. 
Recall that this is given by a quotient of the product $X \times Y$ by the equivalence relation generated by $(x \cdot m, y) \sim (x, m \cdot y)$. We now have $([\ell], [m]) = ([\ell], m \cdot [1]) \sim ([\ell] \cdot m, [1]) = ([\ell f^X(m)],[1])$ and thus in this case we need only consider the first component of the product.
Then if ${\sim}$ is the equivalence relation on $L$ generated by $\ell f^X(m) \sim \ell' f^X(m')$ and $\ell \sim \ell'$ for $\ell \sim^X \ell'$ and $m \sim^Y m'$,
it can be shown that the tensor product $X \otimes Y$ is given by the set $L/{\sim}$ with its inherited left $L$-set structure and a right $N$-set structure given by $f^X \circ f^Y\colon N \to L$.

We can also describe biset homomorphisms in terms of this data. Let $h$ be a $(M,N)$-biset homomorphism from a biset $Y$ defined by $({\sim}, f)$ to a biset $Y'$ defined by $({\sim}', f')$.
Observe that $h([m]) = h(m \cdot [1]) = m \cdot h([1])$. Thus, $h$ is defined by where it sends $[1] \in Y$. Let $\eta \in M$ be such that $h([1]) = [\eta]$. Then $h$ can be recovered form $\eta$ by $h([m]) = [m \eta]$, with the same map being obtained from $\eta_1$ and $\eta_2$ whenever $\eta_1 \sim' \eta_2$. For $h$ to be well-defined we must require that $m \eta \sim' m' \eta$ whenever $m \sim m'$.
Finally, the requirement that $h$ is a biset homomorphism imposes the condition on $\eta$ that $[f(n) \eta] = h([f(n)]) = h([1] \cdot n) = h([1]) \cdot n = [\eta] \cdot n = [\eta f'(n)]$ for all $n \in N$.

Finally, we discuss composition of these biset homomorphisms. Note that if $h$ and $h'$ are (vertically) composable $(M,N)$-biset homomorphisms, then the composite $hh'$ can be represented by the product $\eta'\eta$ of the corresponding elements of $N$ (in reverse). On the other hand, if $h_1\colon X \to X'$ is an $(L,M)$-biset homomorphism and $h_2\colon Y \to Y'$ is an $(M,N)$-biset homomorphism, then in terms of the usual description of the tensor product we have that $h_1 \otimes h_2$ sends $[1] = [([1], [1])]$ to $[([h_1(1)],[h_2(1)])] = [([\eta_1], [\eta_2])] = [([\eta_1 f^{X'}(\eta_2)],[1])]$. Thus, $\eta_1 f^{X'}(\eta_2)$ is an element representing the horizontal composite $h_1 \otimes h_2$.

One could say we have shown that $\LQMod$ is equivalent to a 2-category of monoids, relaxed monoid homomorphisms and equivalence classes of monoid elements.

We are now in a position to fruitfully apply the Grothendieck--Bénabou correspondence to the case of weakly Schreier extensions.
The data of the resulting lax functor can be given by the following.
\begin{enumerate}[1)]
 \item An object of $\LQMod$ for the single object of $\B H$ --- that is, a monoid $N$.
 \item A function from $H$ to the objects of $\LQMod(N,N)$.
       This can be defined by a left $N$-set congruence ${\sim}^h$ on $N$ for each $h \in H$
       and a function $\phi\colon H \times N \to N$ such that $\phi(h, 1) \sim^h 1$ and $\phi(h,nn') \sim^h \phi(h,n)\phi(h,n')$.
       We write $N^h$ for the $(N,N)$-set defined by ${\sim}^h$ and $\phi(h,-)$.
 \item A requirement that ${\sim}^1$ is the equality relation and $\phi(1,n) = n$.
 \item A 2-morphism in $\LQMod$ for each $h,h' \in H$ from $N^h \otimes_N N^{h'}$ to $N^{hh'}$.
       This can be specified by an element $\chi(h,h') \in N$ such that $\phi(h,\phi(h',n)) \chi(h,h') \sim^{hh'} \chi(h,h')\phi(hh',n)$
       and such that $n_1 \phi(h,n'_1) \chi(h,h') \sim^{hh'} n_2 \phi(h,n'_2) \chi(h,h')$ and $n_1 \chi(h,h') \sim^{hh'} n_2 \chi(h,h')$ whenever $n_1 \sim^{h} n_2$ and $n'_1 \sim^{h'} n'_2$.
\end{enumerate}
These must additionally satisfy the following conditions.
\begin{enumerate}[a)]
 \item The naturality condition is trivial.
 \item Secondly, $\chi(1,h) \sim^h 1$ and $\chi(h,1) \sim^h 1$.
 \item Finally, we require $\chi(x,y) \chi(xy,z) \sim^{xyz} \phi(x,\chi(y,z)) \chi(x,yz)$.
\end{enumerate}
The maps $\phi$ and $\chi$ both involve some arbitrary choices, but these will be taken care of once we consider the isomorphism classes of extensions below.

Let us consider how to describe the morphisms of extensions.
A morphism of extensions from the weakly Schreier extensions defined by the data $({\sim}_1,\phi_1,\chi_1)$ to the one defined by $({\sim}_2,\phi_2,\chi_2)$ is given by (expanding the data of a monoidal natural transformation)
\begin{itemize}
 \item a 2-morphism in $\LQMod$ from $N^h_1$ to $N^h_2$ for each $h \in H$. We can specify this by a function $\psi\colon H \to N$ such that $\phi_1(h,n) \psi(h) \sim_2^h \psi(h) \phi_2(h,n)$
       and such that $n\psi(h) \sim_2^h n'\psi(h)$ whenever $n \sim_1^h n'$.
\end{itemize}
This is subject to the conditions that
\begin{itemize}
 \item we have $\psi(1) = 1$ (compatibility with the unit),
 \item and $\chi_1(h,h') \psi(hh') \sim_2^{hh'} \psi(h) \phi_2(h, \psi(h')) \chi_2(h,h')$ (compatibility with the tensor).
\end{itemize}
There are some arbitrary choices in the definition of $\psi$ and so two maps $\psi$ and $\psi'$ define the same morphism of extensions whenever
\begin{itemize}
 \item we have $\psi(h) \sim_2^h \psi'(h)$ for all $h \in H$.
\end{itemize}

It is now not hard to see that a morphism from $({\sim}_2,\phi_2,\chi_2)$ to $({\sim}_1,\phi_1,\chi_1)$ defined by $\psi^*$ is inverse to the 2-morphism given by $\psi$ if and only if
$\psi(h)\psi^*(h) \sim_1^h 1$ and $\psi^*(h)\psi(h) \sim_2^h 1$.

Thus, isomorphism classes of extensions correspond to equivalence classes of data $({\sim}, \phi, \chi)$ under the equivalence relation $\approx$ defined so that
$({\sim}_1, \phi_1, \chi_1) \approx ({\sim}_2, \phi_2, \chi_2)$ whenever there exist maps $\psi, \psi^*$ as above defining an inverse pair of morphisms between them.

This characterisation of weakly Schreier extensions is almost precisely the same as the one presented in \cite{faul2021survey} (and first given in \cite{fleischer1981monoid}), which was derived by other means.
In fact, 9 of the 11 conditions given there on the triples $({\sim}, \phi, \chi)$ coincide directly with our conditions given above.
We now show that the remaining two conditions given there correspond to our single remaining condition.
The two conditions, numbers (4) and (5) of \cite[Section 4.3]{faul2021survey}, are
\begin{equation}\label{eq:one}
 n_1 \sim^{h} n_2 \implies n_1 \phi(h,n') \sim^{h} n_2 \phi(h,n')
\end{equation}
and
\begin{equation}\label{eq:two}
 n'_1 \sim^{h'} n'_2 \implies \phi(h,n'_1) \chi(h,h') \sim^{hh'} \phi(h,n'_2) \chi(h,h').
\end{equation}
We claim that in the presence of the other conditions these are equivalent to our condition
\begin{equation}\label{eq:our}
 n_1 \sim^{h} n_2 \,\land\, n'_1 \sim^{h'} n'_2 \implies n_1 \phi(h,n'_1) \chi(h,h') \sim^{hh'} n_2 \phi(h,n'_2) \chi(h,h').
\end{equation}

\Cref{eq:two} follows immediately from \cref{eq:our} by taking $n_1 = n_2 = 1$.
On the other hand, taking $n'_1 = n'_2 = n'$ and $h' = 1$ in \cref{eq:our} gives $n_1 \phi(h,n') \chi(h,1) \sim^{h} n_2 \phi(h,n') \chi(h,1)$ for $n_1 \sim^{h} n_2$.
But we also have $\chi(h,1) \sim^h 1$ and hence $n_{1,2} \phi(h,n') \chi(h,1) \sim^h n_{1,2} \phi(h,n')$, since ${\sim}^h$ is a left $N$-set congruence.
Then \cref{eq:one} follows by transitivity.

Conversely, \cref{eq:two} together with the left $N$-set property (number (2) in \cite{faul2021survey}) gives that $n'_1 \sim^{h'} n'_2$ implies $n_{1,2} \phi(h,n'_1) \chi(h,h') \sim^{hh'} n_{1,2} \phi(h,n'_2) \chi(h,h')$. Then \cref{eq:one} and another condition (number (3) in \cite{faul2021survey}) give that $n_1 \sim^{h} n_2$ implies $n_1 \phi(h,n') \chi(h,h') \sim^{hh'} n_2 \phi(h,n') \chi(h,h')$.
These yield \cref{eq:our} by transitivity.

Finally, the data characterising morphisms of extensions are agrees with that given in \cite{faul2021survey}. Though, for checking if $\psi$ gives invertible morphism it gives a simplified condition involving left and right invertibility modulo the equivalence relations.

\begin{remark}
 Restricting to the case that $N$ is an abelian group yields the cohomological characterisations of extensions and their morphisms given in \cite{faul2021baer} and \cite{faul2021quotients}.
\end{remark}

Finally, we can apply the generalised Grothendieck construction to the data $({\sim},\phi,\chi)$ as defined above to obtain the corresponding weakly Schreier extension.
Let $L\colon \B H \to \LQMod$ be the lax functor defined by the data. The middle monoid in the extension is given by morphisms of the resulting one-object category ${\int} L$
and consists of pairs $(h,[n])$ where $h \in H$ and $[n] \in L(h) = N/{\sim}^h$.
Multiplication is given by
\begin{align*}
 (h,[n]) \cdot (h',[n']) &= (hh', (\gamma^L_{h,h'})([(n,n')])) \\
                     &= (hh', (\gamma^L_{h,h'})([(n\phi(h,n'),1)])) \\
                     &= (hh', [n\phi(h,n')\chi(h,h')]).
\end{align*}
As before, the cokernel map is given by the projection $(h,[n]) \mapsto h$ and the kernel is defined by $n \mapsto (1,[n])$.
Again, this recovers the usual construction of an extension from the data $({\sim}, \phi, \chi)$ as in \cite{faul2021survey}.

\bibliographystyle{abbrv}
\bibliography{bibliography}
\end{document}